\documentclass{article}
\usepackage[margin=1.25in]{geometry}
\usepackage{amsmath}
\usepackage{amssymb}
\usepackage{amsthm}
\usepackage[all]{xy}
\usepackage{hyperref}
\usepackage{makeidx}
\usepackage{mathtools}
\usepackage[bbgreekl]{mathbbol}

\providecommand{\one}{1}

\providecommand{\Bcom}{B_{\textnormal{com}}}
\providecommand{\Hom}{\textnormal{Hom}}

\providecommand{\Z}{\mathbb{Z}}

\providecommand{\R}{\mathbb{R}}

\providecommand{\RP}{\mathbb{RP}}

\providecommand{\colim}{\mathop{\textnormal{colim}}}
\providecommand{\hocolim}{\mathop{\textnormal{hocolim}}}
\providecommand{\Bcom}{B_{\mathrm{com}}}
\providecommand{\Ecom}{E_{\mathrm{com}}}
\providecommand{\AfCom}{{\mathrm{Af{}Com}}}
\providecommand{\Ab}{{\mathrm{Ab}}}
\providecommand{\AbCo}{{\mathrm{AbCo}}}
\providecommand{\mAbCo}{{\mathrm{mAbCo}}}

\providecommand{\Set}{{\mathsf{Set}}}
\providecommand{\Top}{{\mathsf{Top}}}
\providecommand{\pt}{{\mathrm{pt}}}
\providecommand{\Sing}{{\mathrm{Sing}}}
\providecommand{\Map}{{\mathrm{Map}}}
\providecommand{\Fun}{{\mathsf{Fun}}}
\providecommand{\op}{{\mathsf{op}}}
\providecommand{\cx}{{\mathsf{cx}}}
\providecommand{\sSets}{{\mathsf{sSets}}}
\providecommand{\cm}{{\mathfrak{c}}}

\newtheorem{theorem}{Theorem}
\newtheorem*{theorem*}{Theorem}
\newtheorem{lemma}{Lemma}
\newtheorem{proposition}[lemma]{Proposition}
\newtheorem{corollary}[lemma]{Corollary}
\newtheorem*{introthm}{Theorem \ref{\introthmref}}

\theoremstyle{definition}
\newtheorem{definition}[lemma]{Definition}
\newtheorem{example}[lemma]{Example}
\newtheorem{remark}[lemma]{Remark}

\title{The Complex of Affinely Commutative Sets}
\author{Omar Antol\'{i}n-Camarena and Bernardo Villarreal}
\date{\today}

\begin{document}

\maketitle

%\tableofcontents

\begin{abstract}
  We show that for some classes of groups $G$, the homotopy fiber $\Ecom G$ of the inclusion of the classifying space for commutativity $\Bcom G$ into the classifying space $BG$, is contractible if and only if $G$ is abelian. We show this both for compact connected Lie groups and for discrete groups. To prove those results, we define an interesting map $\cm \colon \Ecom G \to B[G,G]$ and show it is not nullhomotopic for the non-abelian groups in those classes. Additionally, we show that $\cm$ is 3-connected for $G=O(n)$ when $n \ge 3$.
\end{abstract}

\section{Introduction}

Let $G$ be a topological group, and let $\Bcom G$ be the classifying space for commutativity defined by Adem, Cohen and Torres-Giese \cite{Ad5}. This space is defined as the geometric realization of the simplicial space of ordered commuting tuples $\Hom(\Z^\bullet, G)$. The inclusions $\Hom(\Z^k,G)\subseteq G^k$ for $k\ge 0$ give rise to a canonical map $\Bcom G\to BG$. The purpose of this paper is to investigate the homotopy fiber of this map which is denoted $\Ecom G$. Whenever $G$ is abelian, $\Bcom G=BG$, and thus $\Ecom G$ is contractible. The natural question to ask here is if the converse holds. In general this is not true, for example $SL(2,\R)$ is non-abelian and yet $\Ecom SL(2,\R)$ is contractible (see Remark \ref{rmk:Ecom-SL2R}). Nonetheless, for compact connected Lie groups $G$, Adem and G\'{o}mez in \cite{Ad1} proved the converse for a smaller variant of $\Ecom G$, namely, the homotopy fiber $\Ecom G_\one$ of the realization of the simplicial inclusion $\Hom(\Z^k,G)_\one\to G^k$, where the subindex $\one$ denotes the connected component of the trivial homomorphism $\one\colon \Z^k\to G$. This in fact answers the question for the groups $G=SU(n),U(n)$ and $Sp(n)$, since in these cases $\Hom(\Z^k,G)$ is path-connected.

Here we extend this to other classes of groups, using a new map $\cm\colon \Ecom G\to B[G,G]$ that we define using an alternative simplicial model for $\Ecom G$. We say that an ordered $(n+1)$-tuple $(g_0,\ldots,g_n)$ is \emph{affinely commutative} if the quotients ${g_i^{-1}}g_{i+1}$ pairwise commute. The space of affinely commutative $(n+1)$-tuples sit inside the space of $n$-simplices of the nerve of the indiscrete category associated to $G$, defining a subsimplicial space $\AfCom_\bullet(G)$, which is naturally isomorphic to the simplicial model of $\Ecom G$ originally defined in \cite{Ad5}. The commutators in $G$ induce a simplicial map $\cm_\bullet\colon\AfCom_\bullet G\to N[G,G]$ to the nerve of the commutator group of $G$. The realization of this map $\cm\colon \Ecom G\to B[G,G]$ has proved to be interesting and useful, as shown in our main results which answer the question posed before for discrete groups and for compact connected Lie groups.

\newcommand{\introthmref}{T1}
\begin{introthm}
Let $G$ be a discrete group. Then $\cm_*\colon \pi_1(\Ecom G)\to [G,G]$ is a surjective homomorphism. In particular, $G$ is abelian if and only if $\pi_1(\Ecom G)=1$.
\end{introthm}

\renewcommand{\introthmref}{ThGLie}
\begin{introthm}
   Let $G$ be a compact connected Lie group. Then the map $\cm\colon\Ecom G\to B[G,G]$ is nullhomotopic if and only if $G$ is abelian.  In particular, $\Ecom G$ is contractible if and only if $G$ abelian.
\end{introthm}

The first section of this paper gives a simplicial complex $\AfCom(G)$ that models the homotopy type of $\Ecom G$ for discrete groups $G$. This complex is more computationally friendly; for example, using computer calculations we can conclude that $\Ecom S_5$ is not an Eilenberg--MacLane space.

In the last subsection, we study the map $\cm$ for the orthogonal groups $O(n)$, showing that for $n\geq 3$, the homotopy fiber of $\cm$ is $2$-connected. The case $n=2$ is quite interesting. It was shown in \cite{lowdim}, that $\Ecom O(2)\simeq\Sigma(S^1\times S^1)$, so that the adjoint of the map $\cm\colon \Ecom O(2)\to BSO(2)$ up to homotopy is a map $S^1\times S^1\to S^1$. 

\renewcommand{\introthmref}{adcomO(2)}
\begin{introthm}
The adjoint of the map $\cm\colon \Ecom O(2)\to BSO(2)$ is homotopic to the product in $S^1$ or its inverse.
\end{introthm}

\subsection*{Acknowledgments}

We'd like to thank Jeffrey Carlson and Dan Ramras for carefully reading an earlier draft and making numerous suggestions that improved the exposition. We'd also like to thank Mike Miller for pointing out \cite{James} to us.

\section{Two models for $\Ecom G$}

We will present first our simplicial complex model for $\Ecom G$ when $G$ is a discrete group, and then the simplicial model which works for arbitrary topological groups. Let us briefly recall the definitions.

Given a topological group $G$, the classifying space for commutativity in $G$, $\Bcom G$, is defined in \cite{Ad5} to be the geometric realization of a certain subsimplicial space of $NG$ (here $NG$ is the usual nerve of $G$ thought of as a one-object topological category). Namely, $\Bcom G$ is the geometric realization of the subsimplicial space $(\Bcom G)_{\bullet}$ of $NG$ whose $n$-simplices are given by
\begin{equation}
  (\Bcom G)_n = \{(g_1, \ldots, g_n) \in G^n : \text{the } g_i \text{ pairwise commute}\},
  \label{bcom-n-simplices}
\end{equation}
so that $(\Bcom G)_n \cong \Hom(\Z^n, G)$.

By its very definition $(\Bcom G)_{\bullet}$ comes with a degree-wise inclusion into $NG$ and the geometric realization of this inclusion is called the canonical map $\Bcom G \to BG$. One can define (the homotopy type of) $\Ecom G$ as the homotopy fiber of this canonical map, and that is all that we will need for the simplicial complex model. In the section where we give our simplicial space model, we will use a more specific definition of (the homeomorphism type of) $\Ecom G$ as the realization of a certain simplicial space, which is isomorphic to our model.

\subsection{As a simplicial complex}

In this section, let $G$ be a discrete group and $\Ab(G)$ be the poset of abelian subgroups of $G$.
From the definition of $\Bcom G$ as the geometric realization of the simplicial set given in (\ref{bcom-n-simplices}), we see that $\Bcom G = \bigcup_{A \in \Ab(G)} BA$, thinking of all of the $BA$ as subspaces of $BG$. Since the collection $\{BA : A \in \Ab(G)\}$ is closed under intersections, this union is also the colimit indexed by the poset $\Ab(G)$, namely $\Bcom G \cong \colim_{A\in\Ab(G)}BA$. The canonical map $\Bcom G\to BG$ is clearly the map out of the colimit induced by the inclusions $BA \to BG$.

Our first step is to show that this colimit is also a homotopy colimit:
\[\colim_{A\in\Ab(G)}BA\simeq\hocolim_{A\in\Ab(G)}BA.\]
The theory of Reedy model structures gives a simple criterion for when colimits of certain shapes are also homotopy colimits ---see \cite{Reedy} for a nice exposition. In particular, the theory applies to colimits indexed by any poset and in that case says that the colimit is also a homotopy colimit if the map
\[\colim _{A'\subsetneq A\in \Ab(G)}BA'\to BA\] is a cofibration,
which again happens because our collection of spaces is closed under intersections, so that $\colim _{A'\subsetneq A\in \Ab(G)}BA' \cong \bigcup _{A'\subsetneq A\in \Ab(G)}BA'$, and thus the map to $BA$ is the inclusion of a subcomplex.

Now we compute $\Ecom G$ as the homotopy fiber of the canonical map $\hocolim_{A\in\Ab(G)}BA \to BG$. To do this, we use that ``homotopy colimits are universal'', which means that given a map $f \colon X \to Y$ the homotopy pullback functor $f^{\ast} \colon \Top/Y \to \Top/X$ preserves homotopy colimits (\cite[Definition 6.1.1.2, Lemma 6.1.3.14]{HTT}). So, taking homotopy pullbacks of the entire diagram along the base-point inclusion $\ast \to BG$ produces a new diagram whose homotopy colimit is $\Ecom G$. Since the homotopy fiber of the inclusion $BA\to BG$ is the discrete space $G/A$, we get that \[\Ecom G\simeq\hocolim_{A\in\Ab(G)} G/A.\]

Next we need a description of the homotopy colimit of the functor $G/(-)\colon \Ab(G)\to \Set$. Thomason's Theorem (\cite{Thomason}) says that $\hocolim G/(-)$ is the nerve of the Grothendieck construction or category of elements of $G/(-)$. It is straightforward from the definitions that the category of elements of $G/(-)$ is actually just the poset $\AbCo(G) := \{gA : A\in Ab(G)\;\text{ and }g\in G\}$ of cosets of abelian subgroups of $G$ ordered by inclusion. This discussion proves the following proposition.

\begin{proposition}\label{Ecom-NAbCo}
  For a discrete group $G$, we have $\Ecom G \simeq |N(\AbCo(G))|$.
\end{proposition}

This is already progress towards making $\Ecom G$ smaller and more manageable. For example, when $G$ is finite, the poset $\AbCo(G)$ is finite, and $|N(\AbCo(G))|$ is a finite simplicial complex. One immediate improvement is to use only the cosets for maximal abelian subgroups and their intersections.

\begin{proposition}\label{mabco}
  Let $\mAbCo(G)$ be the subposet of $\AbCo(G)$ whose elements are cosets of arbitrary intersections of maximal abelian subgroups. Then there is a homotopy equivalence $|N(\mAbCo(G))| \simeq |N(\AbCo(G))|$.
\end{proposition}

\begin{proof}
  We apply Quillen's Theorem A \cite{Quillen} which says, in the special case of an inclusion of posets, that  $\mAbCo(G) \hookrightarrow \AbCo(G)$ induces a homotopy equivalence provided that for all $C \in \AbCo(G)$, the poset $P_C := \{C' \in \mAbCo(G) : C \subseteq C'\}$ has contractible nerve.

  Let $A$ be an abelian subgroup of $G$ and let $M$ be the intersection of all maximal abelian subgroups of $G$ that contain $A$. Then if $M' \supseteq A$ is any intersection of some maximal abelian subgroups of $G$, we have $M \subseteq M'$. This implies that the poset $P_{gA}$ has $gM$ as least element, and therefore has contractible nerve.
\end{proof}

But we can find an even smaller model, for which we will use the next definition.

\begin{definition}
Let $G$ be a group. We say that a finite subset $S=\{s_0,\ldots,s_n\}$ of $G$ is \emph{affinely commutative} if any of the following equivalent conditions hold:
\begin{enumerate}
\item The consecutive quotients $s_0^{-1}s_1,s_1^{-1}s_{2},\ldots,s_{n-1}^{-1}s_{n}$ pairwise commute.
% \item $\langle s_0^{-1}s_1,s_0^{-1}s_{2},\ldots,s_0^{-1}s_{n}\rangle $ is abelian.
\item $\langle s_i^{-1}s_j:s_i,s_j\in S\rangle$ is abelian.
\item $S$ is contained in a single coset of some abelian subgroup of $G$.
\end{enumerate}
\end{definition}

% \begin{remark}
%   A little more generally, given elements $s_0, s_1, \ldots, s_n \in G$ we can ask how many of the quotients $s_i^{-1} s_j$ need to pairwise commute in order for all of them pairwise commute? If $\Gamma$ is a \emph{connected} graph with vertex set $V(\Gamma) = \{0, 1, \ldots, n\}$ and we ask for the quotients corresponding to edges of $\Gamma$, namely for $\{s_i^{-1} s_j : \{i,j\} \in E(\Gamma)\}$, to commute pairwise, then all of the quotients pairwise commute. The first and second conditions in the definition are the cases where $\Gamma$ is the path $0 \to 1 \to \cdots \to n$ and where $\Gamma$ is the star with center $0$, respectively.
% \end{remark}

We will use the Nerve Theorem to give a simpler description of $|N(\AbCo(G))|$. Recall the statement:

\begin{theorem*}[The Nerve Theorem]
  Let $\mathcal{A}$ be a cover of a CW-complex $X$ by subcomplexes. If every intersection of a finite non-empty subset of $\mathcal{A}$ is either empty or contractible, then $X$ is homotopy equivalent to the simplicial complex with vertices $\mathcal{A}$ and simplices given by finite non-empty subsets of $\mathcal{A}$  with non-empty intersection.
\end{theorem*}

For every $g\in G$, let $X_g= \{C \in \AbCo(G) : g \in C\} = \{gA: A\in \Ab(G)\}$; these sub-posets form a cover of $\AbCo(G)$, so that the geometric realization of the nerves, $\{|N(X_g)|: g\in G\}$ is a cover by subcomplexes of $|N(\AbCo(G))|$.

\begin{lemma}\label{lem:AfCom}
Let $S\subseteq G$ be a finite non-empty subset and $X_S=\bigcap_{s\in S} X_s$. Then 
\begin{enumerate}
\item $X_S$ is nonempty if and only if $S$ is affinely commutative;
\item When $S$ is affinely commutative, $X_S=\{sA : \langle s^{-1}S\rangle\subseteq A\in \Ab(G)\}$, for any $s\in S$. Moreover, in this case, $N(X_S)\simeq \pt$.
\end{enumerate}
\end{lemma}

\begin{proof}
  A coset $C\in \AbCo(G)$ belongs to $X_S$ if and only if for each $s \in S$, we have $s \in C$. This shows that $X_S$ is non-empty if and only if $S$ is affinely commutative. Now let $s$ be any element of $S$. If $S \subseteq C \in \AbCo(G)$, then $C$ is a coset of the abelian group $s^{-1}C$. Thus $S \subseteq C \iff s^{-1}S \subseteq s^{-1}C \iff \langle s^{-1}S \rangle \subseteq s^{-1}C$, which proves the formula for $X_S$ in the affinely commutative case and also shows that $X_S$ has the coset $s \langle s^{-1}S \rangle$ as a minimum element, making its nerve contractible. 
\end{proof}

\begin{definition}
  For a group $G$, $\AfCom(G)$ denotes the simplicial complex whose vertices are the elements of $G$ and whose simplices are the affinely commutative subsets of $G$.
\end{definition}

\begin{remark} \label{maxfaces}
  Since a subset of $G$ is affinely commutative if and only if it is contained in a coset of an abelian subgroup of $G$, the maximal simplices in $\AfCom(G)$ are precisely the cosets of the maximal abelian subgroups of $G$. Also, note that the 1-skeleton of $\AfCom(G)$ is the complete graph with vertex set $G$, because every set of one or two elements is affinely commutative.
\end{remark}

\begin{proposition}\label{prop: AfcomEcomDis}
Let $G$ be a discrete group. Then $\Ecom G\simeq|\AfCom(G)|$.
\end{proposition}

\begin{proof}
The Nerve Theorem and Lemma \ref{lem:AfCom} imply that $N(\AbCo(G))$ is homotopy equivalent to the simplicial complex whose vertices are the $X_g$ and where $\{X_{g_0}, \ldots, X_{g_n}\}$ is an $n$-simplex if and only if $X_{g_0}\cap\cdots \cap X_{g_n}\ne\emptyset$, that is, if and only if $\{g_0,\ldots,g_n\}$ is affinely commutative.
\end{proof}

We can use the following standard presentation for the fundamental group of a simplicial complex to obtain a presentation of the fundamental group of $\Ecom G$.

\begin{lemma}\label{pi1-simplicial-complex}
  Let $K$ be a connected simplicial complex. Then the fundamental group of $K$ has the following presentation where the generators correspond to edges of $K$ and the relations to triangles in $K$:
\[\begin{aligned}
    \pi_1(K) = \langle x_{u,v} : \{u,v\} \in K \mid\; &
    x_{v,v}=1,\;x_{v,u}=x_{u,v}^{-1},\\
    & x_{u,v}x_{v,w}=x_{u,w} \;\mathrm{for}\;\{u,v,w\}\in K\rangle.
  \end{aligned}\]
\end{lemma}

\begin{corollary}\label{pi1-Ecom}
  For a discrete group $G$, $\pi_1(\Ecom G)$ has the following presentation with generators for all pairs of group elements and relations coming from affinely commutative triples:
\[\begin{aligned}
    \pi_1(\Ecom G) = \langle x_{g,h} :\; &
    x_{g,g}=1,\;x_{g,h}=x_{h,g}^{-1},\text{ and }x_{g,h}x_{h,k}=x_{g,k} \\
  & \text{ when }\{g,h,k\}\text{ is affinely commutative}\rangle.
\end{aligned}\]
\end{corollary}

To prove one of our main results we need the following technical and possibly surprising lemma.

\begin{lemma}\label{lem:afftrip}
Let $\{g,h,k\}$ be an affinely commutative triple. Then $[g,h][h,k]=[g,k]$.
\end{lemma}

(Our convention for the commutator is that $[x,y]=x^{-1}y^{-1}xy$.)

\begin{proof}
By hypothesis, $[g^{-1}h,h^{-1}k]=[h^{-1}g,k^{-1}h]=e$, and thus \[[g,h][h,k] = g^{-1}(h^{-1}g)(k^{-1}h)k = g^{-1}(k^{-1}h)(h^{-1}g)k = g^{-1}k^{-1}gk.\]
\end{proof}

\begin{theorem}\label{T1}
The map $\pi_1(\Ecom G)\to [G,G]$ generated by $x_{g,h}\mapsto [g,h]$ is a surjective homomorphism. In particular, $G$ is abelian if and only if $\pi_1(\Ecom G)=1$.
\end{theorem}

\begin{proof}
  We claim that $x_{g,h} \mapsto [g,h]$ is a well defined homomorphism. Lemma \ref{lem:afftrip} says that the commutator preserves the relation $x_{g,h}{x}_{h,k}=x_{g,k}$ for an affine commutative triple $g,h,k$. The remaining relations are obviously preserved by the commutator. The image of this homomorphism includes all commutators, and thus it is surjective.
\end{proof}

\begin{remark}
  As a curiosity, we can use Theorem \ref{T1} to produce a statement that is equivalent to the Feit--Thompson Theorem. Let $\phi_G$ be the homomorphism between the abelianizations of $\pi_1(\Ecom G)$ and $[G,G]$ induced by the homomorphism appearing in Theorem \ref{T1}. The Feit--Thompson Theorem is equivalent to the claim that $\phi_G$ is nonzero for all non-abelian groups $G$ of odd order. Indeed, by Theorem \ref{T1}, $\phi_G$ is always surjective, so if $\phi_G$ is ever zero for some $G$, then $[G,G]/[[G,G],[G,G]] = 1$ and thus $G$ cannot be solvable. On the other hand, if $\phi_G$ is nonzero for all non-abelian groups of odd order, then $[G,G]/[[G,G],[G,G]] \neq 1$ for all non-abelian groups of odd order, so an easy induction argument proves they are all solvable.
\end{remark}

\subsubsection{The homotopy type of $\Ecom G$ for some finite groups $G$}
  A distinct advantage of the simplicial complex $\AfCom(G)$ is its amenability to computer calculations, at least when $G$ is a finite group. For example, one might ask when $\Ecom G$ is an Eilenberg--MacLane space for a finite group $G$, as the authors of \cite{Ad5} asked about $\Bcom G$. The first examples of non-Eilenberg--MacLane $\Bcom G$ were found by Okay \cite[Section 8]{Okay}, who showed that for either of the extraspecial groups of order 32, one has $\pi_1(\Bcom G) \cong G \times \Z/2$ and $\pi_2(\Bcom G) \cong \Z^{151}$ (!). In the course of working out that example, Okay shows that for those groups, $\Ecom G$ isn't an Eilenberg--MacLane space either---instead it has $\pi_1$ and $\pi_2$ given by $\Z/2$ and $\Z^{151}$---and in fact he shows that there is a homotopy fiber sequence
  \[ \bigvee^{151} S^{2} \to \Ecom G \to B(\Z/2,1).\]

  By using either the presentation of $\pi_1(\Ecom G)$ in Corollary \ref{pi1-Ecom} or even the definition of $\AfCom(G)$ directly, we can easily find more groups for which $\Ecom G$ is not an Eilenberg--MacLane space. Indeed, if for a finite group $G$ we see that $\pi_1(\Ecom G)$ has torsion, then $\Ecom G$ cannot be a $K(\pi,1)$, since the torsion element would force it to have non-zero cohomology in arbitrarily high degrees but $\AfCom(G)$ is a finite simplicial complex. Using SageMath \cite{Sage}, which has excellent support for finite simplical complexes, one can define $\AfCom$ as follows:

\begin{verbatim}
abelian_cosets = lambda G: [C for H in G.subgroups() if H.is_abelian()
                              for C in G.cosets(H)]
AfCom = lambda G: SimplicialComplex(abelian_cosets(G), maximality_check=True)
\end{verbatim}

After giving that definition, \verb|AfCom(G).fundamental_group()| produces an explicit presentation of $\pi_1(\AfCom(G))$, and one can also compute $H_d(\Ecom G; \Z)$ with \verb|AfCom(G).homology(dim=d)|.

For example, inspecting the results of \verb|AfCom(SymmetricGroup(5)).fundamental_group()| shows that $\pi_1(\Ecom S_5)$ contains elements of order two, and thus $\Ecom S_5$ is not an Eilenberg--MacLane space. In this way, we also corroborated that $\Ecom G$ is not a $K(\pi,1)$ when $G$ is either of the extraspecial groups of order $32$. Both the fundamental group and homology computations seem to be very memory intensive so this method doesn't work in practice except for fairly small groups.

\bigskip

For groups with very few abelian subgroups, it is not too hard to compute the homotopy type of $\Ecom G$ by hand using either $|\AfCom(G)|$ or $|N(\mAbCo(G))|$. Here is an example of each.

\begin{example}
  We have $\Ecom S_3 \simeq \bigvee^8 S^1$. Indeed, by Remark \ref{maxfaces}, $\AfCom(S_3)$ has as 1-skeleton the complete graph on 6 vertices and beyond that has only two simplices, which are disjoint triangles corresponding to the two cosets of $A_3$. The union of those two triangles and an edge connecting them is contractible, collapsing that subcomplex to a point leaves a wedge of $8$ circles.
\end{example}

\begin{example}
  Let $Q_{2^n}$ be the generalized quaternion group of order $2^n$. It is not hard to show that its maximal abelian subgroups are as follows: there is one cyclic subgroup $M$ of order $2^{n-1}$, and every $x \notin M$ generates a cyclic subgroup of order $4$. The intersection of any two of these maximal abelian subgroups is the center $Z$ of $Q_{2^n}$, which is of order $2$. Each of those cyclic groups of order $4$ is generated by two elements that do not belong to $M$: if $x$ is a generator, then $x^3$ is the other. In particular, there are $2^{n-2}$ such maximal abelian subgroups of order $4$.

  This means that the poset $\mAbCo(Q_{2^n})$ has rank 2, the minimal elements are the $2^{n-1}$ cosets of $Z$, and all other elements are maximal: the $2$ cosets of $M$ and the $2^{n-2}$ cosets of each of the $2^{n-2}$ subgroups or order $4$ generated by an element not in $M$.
  
  Because the poset $\mAbCo(Q_{2^n})$ only has rank 2, its nerve is a graph, namely, the nerve is the same as the Hasse diagram, and thus $N(\mAbCo(Q_{2^n}))$ is homotopy equivalent to a wedge of circles. The number of circles is given by $e - v + 1$, where $v$ and $e$ are the number of vertices and edges in the graph. We've listed the vertices, so we find that $v = 2^{n-1} + 2 + 2^{n-2} \cdot 2^{n-2} = 2^{2n-4} + 2^{n-1} + 2$. Now we count the edges: each coset of $M$ contains $2^{n-2}$ cosets of $Z$, these provide $2 \cdot 2^{n-2}$ edges; each coset of a cyclic group of order 4 generated by an element not in $M$ contains $2$ cosets of $Z$, contributing a further $2^{n-2} \cdot 2^{n-2} \cdot 2$ edges. Thus we get a total of $e = 2^{2n-3}+2^{n-1}$ edges, and therefore $\Ecom Q_{2^n}$ has the homotopy type of a wedge of $e-v+1 = 2^{2n-4} - 1$ copies of $S^1$.
\end{example}

\subsection{The simplicial model}

Let $G$ be a topological group. First we recall the simplicial model for $\Ecom G$ given in \cite{Ad5}. For every $n$, $(\Ecom G)_n=G\times \Hom(\Z^n,G)\subseteq G^{n+1}$; the face maps are given by
\[d_i(g_0,\ldots,g_n)=\left\{\begin{matrix}
(g_0,..,g_{i}g_{i+1},\ldots,g_n)&i<n\\
(g_0,\ldots,g_{n-1})&i=n
\end{matrix}
\right.\]
and for every $0\leq i\leq n$ the degeneracy maps are $s_i(g_0,\ldots,g_n)=(g_0,\ldots,g_{i},e,g_{i+1},\ldots,g_n)$.  Notice that this makes $(\Ecom G)_\bullet$ a simplicial subspace of the simplicial model for $EG$ arising from the bar construction.

For every $n\geq 0$, consider the subspaces 
\[\AfCom_n(G):=\{(g_0,\ldots,g_n) \text{ is affinely commutative}\}\subseteq G^{n+1}.\]
These spaces assemble into a simplicial space, $\AfCom_\bullet(G)$, where the face maps are given by $d_i(g_0,\ldots,g_n)=(g_0,\ldots,\widehat{g_i},\ldots,g_n)$ and the degeneracy by $s_i(g_0,\ldots,g_n)=(g_0,\ldots,g_i,g_i,\ldots,g_n)$.

\begin{proposition}\label{prop:AfcomEcom}
Let $G$ be a topological group. Then the simplicial spaces $(\Ecom G)_\bullet$ and $\AfCom_\bullet$ are isomorphic, in particular, $\Ecom G \cong |\AfCom_\bullet(G)|$.
\end{proposition}

\begin{proof}
Let $\overline{G}$ denote the indiscrete category associated to $G$, that is, the space of objects is $G$ and the space of morphisms is $G\times G$. Consider the nerve $N\overline{G}$. Since there is a unique morphism between every pair of objects in $\overline{G}$, an $n$-simplex in $N\overline{G}$ can be given simply by listing $n+1$ objects, $(g_0, \ldots, g_n)$. There is a simplicial isomorphism $\varphi\colon N\overline{G}\to (EG)_\bullet$ given at each level $n$ by
\[(g_0,\ldots,g_n)\mapsto (g_0,g_0^{-1}g_1,\ldots,g_{n-1}^{-1}g_n).\]
By definition of affine commutativity, $\varphi(\AfCom_n(G))=G\times\Hom(\Z^n,G)$, and thus $\AfCom_\bullet (G)$ is a simplicial model for $\Ecom G$.
\end{proof}

We now present an alternative proof of Proposition \ref{prop:AfcomEcom} for a discrete group $G$. First, for a simplicial complex $K$, we define the simplicial set $\Sing(K)$ by setting
\[\Sing(K)_n:=\{(x_0,\ldots,x_n)\in V(K)^{n+1}:\{x_0,\ldots,x_n\}\in K\}, \]
and letting face and degeneracy maps be given by deleting and repeating the coordinates as usual. The notation is chosen by analogy with the singular simplicial set of a topological space. For a space $X$, the set of $n$-simplices $\Sing(X)_n$ is the set of continuous maps $|\Delta^n| \to X$ from the topological $n$-simplex to $X$. Analogously, for a simplicial complex $K$, the set $\Sing(K)_n$ can be described as consisting of simplicial maps from the simplicial complex $\Delta_\cx^n$ (consisting of all subsets of a set with $n+1$ elements) to $K$.

\begin{proposition}\label{folklore}
  Let $K$ be a simplicial complex. Then $|\Sing(K)| \simeq |K|$.
\end{proposition}

\begin{remark}
  This homotopy equivalence is far from being a homeomorphism. For example, for any $n>0$, $|\Sing(\Delta_\cx^n)|$ is homeomorphic to the infinite-dimensional sphere $S^{\infty}$.
\end{remark}

Proposition \ref{folklore} seems to be well-known---at least, both authors knew it and several algebraic topologists we asked knew it as well. However we were unable to locate a reference and thus decided to include a couple of proofs. But first, let's point out how to get Proposition \ref{prop:AfcomEcom} from it.

\begin{proof}[Alternative proof of Proposition \ref{prop:AfcomEcom}]
The key observation is that $\AfCom_\bullet(G) = \Sing(\AfCom(G))$. Then, $|\AfCom_\bullet(G)|\simeq|\AfCom(G)|$ by Proposition \ref{folklore}, and the result follows from Proposition \ref{prop: AfcomEcomDis}.
\end{proof}

In order to prove Proposition \ref{folklore}, let's introduce another standard way of making a simplicial set from a simplicial complex. Pick an arbitrary total order $\le$ on the vertices of $K$ (or even just a partial order whose restriction to each face of $K$ is total). Then we can define $\Sing_{\le}(K)$ to be the simplicial subset of $\Sing(K)$ given by
\[\Sing_{\le}(K)_n:=\{(x_0,\ldots,x_n)\in V(K)^{n+1}:\{x_0,\ldots,x_n\}\in K, x_0 \le \cdots \le x_n\}.\]

It is well-known and easy to see that the geometric realization of $\Sing_{\le}(K)$ is \emph{homeomorphic} to the geometric realization of $K$, so to prove Proposition \ref{folklore} it suffices to prove that we have a homotopy equivalence $\Sing_{\le}(K) \simeq \Sing(K)$. We feel the basic intuition behind Proposition \ref{folklore} is the following easily-verified relationship between $\Sing_{\le}(K)$ and $\Sing(K)$: to obtain $\Sing(K)$ from $\Sing_{\le}(K)$ we replace each $n$-simplex with $\Sing(\Delta_\cx^n)$, or more precisely,
\[ \Sing_\le(K) \cong \colim_{\Delta^n \to \Sing_{\le}(K)} \Delta^n, \quad\text{while}\quad \Sing(K) \cong \colim_{\Delta^n \to \Sing_{\le}(K)} \Sing(\Delta_\cx^n).\]

If we knew both colimits were in fact homotopy colimits, then Proposition \ref{folklore} would follow, since both $\Delta^n$ and $\Sing(\Delta_\cx^n)$ are contractible. For the first colimit this is a standard fact, and what our proofs do, morally, is to verify this for the second colimit.

Note that the simplicial sets $E[n] := \Sing(\Delta_\cx^n)$ depend functorially on $[n] \in \mathbb{\Delta}$ and thus give a cosimplicial simplicial set $E \colon \mathbb{\Delta} \to \sSets$. We can also describe $E[n]$ as the nerve of the indiscrete category with $n+1$ objects (which already shows it is contractible), or as the $0$-coskeleton of the discrete simplicial set with $n+1$ vertices.

\begin{proof}[Proof of Proposition \ref{folklore} using Reedy model structures]
  Let $\Delta \colon \mathbb{\Delta} \to \sSets$ be the canonical cosimplicial simplicial set (the Yoneda embedding). In terms of the functor tensor product,
  \[\otimes_{\mathbb{\Delta}} : \Fun(\mathbb{\Delta}, \sSets) \times \Fun(\mathbb{\Delta}^\op, \mathsf{sSets}) \to \sSets,\]
  we can write the above colimits as $\Sing(K) \cong E \otimes_{\mathbb{\Delta}} \Sing_\le(K)$ and, of course, $\Sing_\le(K) \cong \Delta \otimes_{\mathbb{\Delta}} \Sing_\le(K)$, where we regard a simplicial set $X$ as a functor $\mathbb{\Delta}^\op \to \sSets$, by regarding the set $X_n$ of $n$-simplices as a discrete or constant simplicial set.

  By \cite[Proposition A.2.9.26]{HTT}, the functor tensor product is a left Quillen bifunctor when we equip $\sSets$ with the Quillen model structure and both functor categories with the corresponding Reedy model structure. Recall that all bisimplicial sets are Reedy cofibrant, so for any simplicial set $X$, the functor $- \otimes_{\mathbb{\Delta}} X$ is left Quillen.

Now, both $\Delta^n$ and $E[n]$ are contractible simplicial sets, so the inclusion $\Delta \to E$ is an object-wise weak equivalence. Thus, we need only check that both $\Delta$ and $E$ are Reedy cofibrant to conclude from Ken Brown's lemma that the left Quillen functor $- \otimes_{\mathbb{\Delta}} X$ will send the inclusion to a weak equivalence, as desired.

For $\Delta$ this is well known: the latching map $L_n \Delta \to \Delta^n$ is readily seen to be the inclusion $\partial \Delta^n \to \Delta^n$, a monomorphism and thus a cofibration in $\mathsf{sSet}$. The case of $E$ is very similar. Indeed, the latching object is given by $L_n(E) = \mathop{\mathrm{colim}}_{[k] \hookrightarrow [n]} E[k]$, which one can check consists of all simplices of $E[n]$ that do not involve all $n+1$ vertices, and the canonical map $L_n(E) \to E$ is then a monomorphism.
\end{proof}

Andrea Gagna remarked on MathOverflow that $\Sing(K)$ seems intrinsically linked to symmetric simplicial sets and mentioned \cite[\S 8.3]{Cis}. From the results there one can easily obtain another proof.

\begin{proof}[Proof of Proposition \ref{folklore} using symmetric simplicial sets]
  Let $\mathbb{\Upsilon}$ be the category of finite non-empty sets and all functions between them. The category of symmetric simplicial sets is defined to be $\mathsf{\Sigma Set} := \mathop{\mathrm{Fun}}(\mathbb{\Upsilon}^{\mathrm{op}}, \mathsf{Set})$. There is an obvious functor $v : \mathbb{\Delta} \to \mathbb{\Upsilon}$, including monotone functions into all functions. That functor (thought of as a functor $\mathbb{\Delta}^{\mathrm{op}} \to \mathbb{\Upsilon}^{\mathrm{op}}$) induces adjunctions $v_{!} \dashv v^{\ast} \dashv v_{\ast}$, where $v^{\ast} : \mathsf{\Sigma Set} \to \mathsf{sSet}$ is precomposition with $v$ and $v_{!}$ and $v_{\ast}$ are left and right Kan extension along $v$.

  We can use these functors to relate $\Sing(K)$ with $\Sing_{\le}(K)$, namely, we have $v^{\ast} v_{!} \Sing_\le(K) \cong \Sing(K)$. More generally, for any simplicial set $X$ we have $v^{\ast} v_{!} X \cong E \otimes_{\mathbb{\Delta}} X$. This is not hard to do directly for $\Sing_\le(K)$. A slicker way is to notice that $E[n] \cong v^{\ast}(\Upsilon_n)$ where $\Upsilon_n$ is the representable symmetric simplicial set corresponding to $[n]$. Thus $E[n] \cong v^{\ast}v_{!}\Delta^n$ and since both $v^{\ast}$ and $v_{!}$ are left adjoints we have
\[E \otimes_{\mathbb{\Delta}} X \cong (v^{\ast} \circ v_{!} \circ \Delta) \otimes_{\mathbb{\Delta}} X \cong v^{\ast} v_{!} (\Delta \otimes_{\mathbb{\Delta}} X) \cong v^{\ast} v_{!} X.\]

Moreover, it is straightforward to check that the canonical map $X \to E \otimes_{\mathbb{\Delta}} X$ used in the first proof corresponds to the unit $X \to v^{\ast} v_{!} X$ under the above isomorphism. It only remains to prove the unit is a weak equivalence.

In \cite[\S 8.3]{Cis}, Cisinski proves there is a model structure on $\mathsf{\Sigma Set}$ for which $(v_{!}, v^{\ast})$ is a Quillen equivalence in which $v^{\ast}$ creates the weak equivalences (that is, a map $f$ of symmetric simplicial sets is a weak equivalence if and only if $v^{\ast}(f)$ if a weak equivalence of simplicial sets). When a right Quillen functor creates weak equivalences, it is a right Quillen equivalence if and only if the components of the (underived) unit of the adjunction at cofibrant objects are weak equivalences. Since all simplicial sets are cofibrant, we obtain that $X \to v^{\ast} v_{!} X$ is a weak equivalence, as desired.
\end{proof}

\section{The commutator map $\cm\colon\Ecom G\to B[G,G]$}

We now introduce our main tool for proving $\Ecom G$ is not contractible for certain non-abelian groups: a map $\cm\colon\Ecom G\to B[G,G]$ which we will show is not nullhomotopic in some cases.

Recall the simplicial model of the classifying space $BG$ given by the nerve $NG$, where $G$ is thought of as a one-object category. The face maps $d_i\colon G^n\to G^{n-1}$ are given by multiplying two adjacent coordinates, except for $d_0$ and $d_n$ which drop the first and last coordinate respectively. The degeneracy maps $s_i\colon G^n\to G^{n+1}$ are given by inserting the identity element in the $i^{\mathrm{th}}$ coordinate.

\begin{lemma}
Let $G$ be a topological group, and let $[G,G]$ be its commutator subgroup. Then the maps $\AfCom_n(G)\to [G,G]^n$ given by $(g_0,\ldots,g_n)\mapsto ([g_0,g_1],\ldots,[g_{n-1},g_n])$ assemble to give a simplicial map $\cm_\bullet\colon\AfCom_\bullet(G)\to N[G,G]$ which geometrically realizes to a map $\cm\colon\Ecom G\to B[G,G]$.
\end{lemma}

\begin{proof}
We need to verify that the above maps commute with degeneracy and face maps. For the degeneracy maps and the cases $i=0,1$ of the face maps, this is immediate; the remaining cases follow from Lemma \ref{lem:afftrip}.
\end{proof}

\begin{remark}
  For a discrete group $G$, this map $\cm$ induces on fundamental groups the homomorphism of Theorem \ref{T1}. Furthermore, when $G$ is discrete, that homomorphism is all there is to this map: since $B[G,G] \simeq K([G,G],1)$, the map $\cm \colon \Ecom G \to B[G,G]$ automatically factors through the first Postnikov section of $\Ecom G$, namely, $B \pi_1(\Ecom G)$. However, if $G$ is not discrete, the map $\cm$ does not in general factor through $B \pi_1(\Ecom G)$. For example, Theorem \ref{adcomO(2)} and Proposition \ref{EcomOn} together imply that $\cm : \Ecom O(n) \to BSO(n)$ is nonzero on $\pi_2$ for $n\ge 2$; and Lemma \ref{lem:pi7su2} says that $\cm : \Ecom SU(2) \to BSU(2)$ is nonzero on $\pi_7$.
\end{remark}

\subsection{The map $\cm$ for compact connected Lie groups}

The simplicial models for $\Ecom G$ and $EG$ have the same $1$-skeleton, so that the first terms of the skeletal filtrations of their realizations are the same, $F_1\Ecom G=F_1EG$. Now, recall that $F_1EG \simeq G*G \simeq \Sigma (G\wedge G)$, and that $F_1BG = \Sigma G$. Let $c\colon G\times G\to [G,G]$ denote the commutator map. Since $[x,y]=1$  if either $x=1$ or $y=1$, it follows that $c$ factors through a map $\tilde{c}\colon G\wedge G\to [G,G]$ (which is called the \emph{reduced commutator}). Then we have a homotopy commutative diagram
\begin{equation}\label{cdiagram}
  \begin{gathered}
\xymatrix{
\Ecom G\ar[r]^{\cm}&B[G,G]\\
\Sigma(G\wedge G)\ar[r]^{\Sigma \tilde{c}}\ar[u]&\Sigma [G,G].\ar[u]
 }
\end{gathered}
\end{equation}
 
 The following classical theorem is the key ingredient to prove one of our main results.
 
 \begin{theorem}[Araki, James, Thomas \cite{AJTh}]  \label{comm}
 Let $G$ be a compact connected Lie group. Then $c\colon G\times G\to G$ is nullhomotopic if and only if $G$ is abelian.
 \end{theorem}
 
 \begin{theorem}\label{ThGLie}
 Let $G$ be a compact connected Lie group. Then the map $\cm\colon\Ecom G\to B[G,G]$ is nullhomotopic if and only if $G$ is abelian.  In particular, $\Ecom G$ is contractible if and only if $G$ abelian.
 \end{theorem}
 \begin{proof}
Suppose $\cm\colon\Ecom G\to B[G,G]$ is nullhomotopic. Diagram (\ref{cdiagram}) implies that the composite $\Sigma(G\wedge G)\xrightarrow{\Sigma \tilde{c}}\Sigma [G,G]\to B[G,G]$ is nullhomotopic as well. Then the adjoint map 
\[G\wedge G\xrightarrow{\tilde{c}}[G,G]\xrightarrow{\simeq}\Omega B[G,G]\]
is also nullhomotopic, and therefore $\tilde{c}$ is nullhomotopic. Since the commutator $c\colon G\times G\to G$ factors through $\tilde{c}$, it is nullhomotopic as well. The result now follows from Theorem \ref{comm}.
 \end{proof}

A similar result was proven by Adem and G\'{o}mez \cite[Corollary 7.5]{Ad1} for a smaller variant of $\Ecom G$: in the first model for $\Ecom G$ that we described, one can consider at each level the connected component of the trivial homomorphism in $\Hom(\Z^n,G)$ and denote it by $\Hom(\Z^n,G)_\one$. Then $\Ecom G_\one=|G\times\Hom(\Z^\bullet,G)_\one|$. Now we can expand on their result (which is the equivalence of the last three parts):
 
\begin{corollary}
 Let $G$ be a compact connected Lie group. The following are equivalent:
 \begin{enumerate}
  \item $\Ecom G$ is contractible;
  \item $\Ecom G_\one$ is contractible;
  \item $\Ecom G_\one$ is rationally acyclic;
  \item $G$ is abelian.
 \end{enumerate}
\end{corollary}
 
 \begin{remark}\label{rmk:Ecom-SL2R}
 The compactness condition is necessary. For example $G=SL(2,\R)$ is non-abelian, but it {\bf is} homotopy abelian: it deformation retracts to $SO(2)$. A result of Pettet and Suoto \cite[Corollary 1.2]{PS}, implies that $\Ecom SL(2,\R)\simeq \Ecom SO(2)=ESO(2)\simeq\pt $.
 \end{remark}

 \subsection{The case $G = SU(2)$}
 
 Computing the values of $\cm$ in homotopy groups or homology can be quite challenging. A first obstacle is that we only have explicit descriptions of the homotopy type of $\Ecom G$ for very few non-abelian groups. In \cite[Theorem 1.4]{lowdim} it is shown that $\Ecom SU(2)\simeq S^4\vee\Sigma^4\RP^2$, so in particular, this $\Ecom G$ happens to be a suspension. Thus the map $\cm \colon \Ecom SU(2) \to B[SU(2),SU(2)] = BSU(2)$ has an adjoint under the suspension-loops adjunction, which belongs to the abelian group $[S^3\vee\Sigma^3\RP^2,S^3]=\Z\oplus[\Sigma^3\RP^2,S^3]$. We claim that the cohomotopy group $\pi^3(\Sigma^3\RP^2):=[\Sigma^3\RP^2,S^3] \cong \Z/4$. To see this, recall that $\Sigma^3\RP^2$ is the homotopy cofiber of the degree 2 map on $S^4$:
\[
\xymatrix{
S^4\ar[r]^{2}\ar[d]&S^4\ar[d]\\
\ast\ar[r]&\Sigma^3\RP^2.
}
\]
 After applying $\Map_{*}(-,S^3)$ we get a long exact sequence:
\[\pi_5(S^3)\xrightarrow{2\times}\pi_5(S^3)\to\pi^3(\Sigma^3\RP^2)\to \pi_4(S^3)\xrightarrow{2 \times}\pi_4(S^3).\]
Both homomorphisms labeled $2\times$ are zero, since $\pi_5(S^3) \cong \pi_4(S^3) \cong \Z/2$. So the long exact sequence reduces to a short exact sequence $0\to\Z/2\to \pi^3(\Sigma^3\RP^2)\to\Z/2\to 0$, implying that $\pi^3(\Sigma^3\RP^2)$ is either $\Z/4$ or $\Z/2\oplus\Z/2$.

The suspension homomorphism $\pi^3(\Sigma^3\RP^2)\to\pi^4(\Sigma^4\RP^2)$ is a surjection, and the latter group is in the stable range, that is, $\pi^4(\Sigma^4\RP^2)\cong \mathbb{S}^0(\RP^2)$ the stable cohomotopy group of $\RP^2$. This group was computed in \cite[Theorem 1.5]{Mukai2} using Kahn-Priddy maps, yielding $\mathbb{S}^0(\RP^2)=\Z/4$. Since $\pi^3(\Sigma^3\RP^2)$ surjects onto $\Z/4$, it must be $\Z/4$ itself.

For the reader's convenience, we will sketch Mukai's argument in the special case we need, to see directly that $\pi^3(\Sigma^3\RP^2)$ contains an element whose double is nonzero, and therefore $\pi^3(\Sigma^3\RP^2) \cong \Z/4$. Let $\eta\colon S^3\to S^2$ be the Hopf map. Then $\Sigma \eta$ has order 2 in $\pi_4(S^3)$, so the outer square in the following diagram commutes up to homotopy:
\[
\xymatrix{
S^4\ar[r]^{2}\ar[d]&S^4\ar[d]\ar@/^/[rdd]^{\Sigma\eta}\\
\ast\ar[r]\ar@/_/[rrd]&\Sigma^3\RP^2\ar@{-->}[rd]^{\phi_3}\\
&&S^3
}.
\]
Choosing such a homotopy determines a map $\phi_3$ (following the notation of \cite{Mukai} for Kahn-Priddy maps). It is well known that the Hopf map in orthogonal topological $K$-theory gives an epimorphism $(\Sigma\eta)^*\colon\Z \cong \widetilde{KO}^3(S^3)\to \widetilde{KO}^3(S^4) \cong \Z/2$. Applying $\widetilde{KO}^3$ to the top triangle of the above diagram, we see that that $(\Sigma\eta)^*$ factors through $\widetilde{KO}^3(\Sigma^3\RP^2) \cong \widetilde{KO}(\RP^2) \cong \Z/4$, which then implies that $\phi_3^*$ must hit a generator in $\Z/4$ ---otherwise $(\Sigma\eta)^*$ would be 0. Therefore the composite $\Sigma^3\RP^2\xrightarrow{\phi_3} S^3\xrightarrow{2} S^3$ is nonzero on $\widetilde{KO}^3$, making $2\phi_3$ non-nullhomotopic. 

We now have that the homotopy class of the adjoint of $\cm$ is represented by a nonzero element in $\Z\oplus \Z/4$. We claim that the component in $\Z$ is non-zero.

\begin{lemma}\label{lem:pi7su2}
  The following composite is a generator of $\pi_7(BSU(2)) \cong \Z/12$: \[S^7 \simeq F_1\Ecom SU(2) \to \Ecom SU(2) \xrightarrow{\cm} BSU(2).\] In particular, $\cm \colon \Ecom SU(2)\to BSU(2)$ is surjective on $\pi_7$.
\end{lemma}
\begin{proof}
 To see this, notice that in this case diagram (\ref{cdiagram}) reads as
\[
\xymatrix{
\Ecom SU(2)\ar[r]^{\cm}&BSU(2)\\
S^7\ar[u]^{j}\ar[r]^{\Sigma{\tilde c}}&S^4\ar[u]^{i}
}
\]
It is known that $[\tilde c]\in \pi_6(S^3)=\Z/12$ is a generator (see \cite[Section 9]{James}). Armed with prior knowledge that $\pi_7(S^4) \cong \Z \oplus \Z/12$, an inspection of the long exact sequence of homotopy groups of the Hopf fibration $S^3\to S^7\to S^4$, shows there is a splitting $\pi_7(S^4)=\pi_7(S^7)\oplus\langle[\Sigma\tilde c]\rangle$. Using again the long exact sequence of homotopy groups but now for the delooping $S^7\to S^4\xrightarrow{i} BSU(2)$ we see that $i_*$ is surjective on $\pi_7$ and that $\ker i_*=\pi_7(S^7)$. Therefore $\pi_7(BSU(2))=\langle[i\circ \Sigma\tilde c]\rangle$, and by commutativity of diagram (\ref{cdiagram}) the lemma follows.
\end{proof}

\begin{corollary}
The $\Z$-component of the element in $\pi^3(\Sigma^3 \RP^2) \cong \Z \oplus \Z/4$ corresponding to the adjoint of $\cm$ is not a multiple of $3$.
\end{corollary}

\begin{proof}
  Lemma \ref{lem:pi7su2} shows the homomorphism $\Z\oplus \Z/4 \cong [\Ecom SU(2),BSU(2)]\xrightarrow{j^{*}} \pi_7(BSU(2)) \cong \Z/12$ maps $[\cm]$ to a generator in $\Z/12$ and is therefore surjective. Since the image of the $\Z/4$ summand is necessarily contained in $3\Z/12$, if the component in $\Z$ of the adjoint of $\cm$ were a multiple of $3$, then $j^{*}$ would have image contained in $3\Z/12$ and could not be surjective.
\end{proof}

\paragraph{Question.} Explicitly, which element of $\Z\oplus \Z/4$ represents the homotopy class of the commutator map $\cm\colon \Ecom SU(2)\to BSU(2)$?

\begin{remark}
  The case of $G = U(2)$ is not really different from that of $SU(2)$, namely, the inclusion $SU(2) \hookrightarrow U(2)$ induces a homotopy equivalence $\Ecom SU(2) \simeq \Ecom U(2)$ and both groups have $SU(2)$ as commutator group; furthermore, the following diagram commutes:
  \[
\xymatrix{
\Ecom U(2)\ar[r]^{\cm}&BSU(2)\\
\Ecom SU(2)\ar[r]^{\cm}\ar[u]^{\simeq}&BSU(2)\ar[u]_{\simeq}.
 }
\]  
\end{remark}

\subsection{The case $G=O(n)$}

We first study the case of the simplest non-abelian disconnected Lie group $G=O(2)$. This is a natural destination after dealing with the cases of discrete groups and of compact connected groups: since $O(2)$ has the feature that both $\pi_0(O(2)) \cong \Z/2$ and the connected component of the identity, $O(2)_0 = SO(2)$, are abelian, we think of $O(2)$ as a small example of a compact Lie group $G$ whose non-abelianness comes from the interaction of $\pi_0(G)$ with $G_0$ and not from either group on its own. We have $[O(2),O(2)]=SO(2)$, and thus a map $\cm\colon\Ecom O(2)\to BSO(2)$. It was shown in \cite{lowdim} that $\Ecom O(2)\simeq \Sigma(S^1\times S^1)$, and thus the adjoint of $\cm$ up to homotopy is a map $S^1\times S^1\to \Omega BSO(2)\simeq S^1$.

\begin{theorem}\label{adcomO(2)}
The adjoint of the map $\cm\colon \Ecom O(2)\to BSO(2)$ is homotopic to the product in $S^1$ or its inverse.
\end{theorem}

To see this it will be enough to show that the induced morphism on $H_2 = H_2(-; \Z)$, or more precisely the composite $\Z \oplus \Z \cong H_2(\Ecom O(2)) \xrightarrow{\cm_{*}} H_2(BSO(2)) \cong \Z$, is given by plus or minus the sum. Indeed the homotopy class of the adjoint map $\cm^\dagger \colon S^1 \times S^1 \to SO(2)$ is determined by its action on $H_1$, and we can recover $\cm$ up to homotopy as the composite $\Ecom O(2) \xrightarrow{\Sigma \cm^\dagger} \Sigma SO(2) \to BSO(2)$---in which the map $\Sigma SO(2) \to BSO(2)$ is an $H_2$-isomorphism. 

 For the homology calculation, we study the reduced commutator $\tilde{c}\colon O(2)\wedge O(2)\to SO(2)$. Let us write $O(2)=SO(2)\sqcup A \; SO(2)$, where $A=
\left(
\begin{matrix}
1&0\\
0&-1 
\end{matrix}
\right)$. Then $O(2) \wedge O(2)$ is given by:
\[\Bigl((SO(2)\wedge SO(2))\vee(SO(2)\wedge A \; SO(2)_+)\vee(A \; SO(2)_+\wedge SO(2))\Bigr) \sqcup(A \; SO(2)\times A \; SO(2)).\]
Now, $SO(2)\wedge A \; SO(2)_+\simeq S^2\vee S^1$ in such a way that the inclusion $S^1 \hookrightarrow S^2 \vee S^1$ is homotopic to $S^1 \simeq SO(2)\times\{A\}\hookrightarrow SO(2)\wedge A \; SO(2)_+$. To describe the suspension of $O(2) \wedge O(2)$, recall that $\Sigma (X \sqcup Y) \simeq \Sigma X \vee \Sigma Y \vee S^1$ and $\Sigma (X \times Y) \simeq \Sigma X \wedge Y \vee \Sigma X \vee \Sigma Y$. All together, we have
$\Sigma O(2) \wedge O(2) \simeq (S^3)^{\vee 4} \vee (S^2)^{\vee 4} \vee S^1$.  The four copies of $S^2$ are given by \[\Sigma(SO(2)\times\{A\})\vee\Sigma(\{A\}\times SO(2))\vee \Sigma(A \; SO(2)\times\{A\})\vee \Sigma(\{A\}\times A \; SO(2)),\] where the last two correspond to the 2-cells of $\Sigma(A \; SO(2)\times A \; SO(2))$. Then $H_2(\Sigma (O(2)\wedge O(2)))=\Z^4$.

We want to understand the effect of the map $\Sigma\tilde{c}$ on $H_2$. To do this, let $R_\theta \in SO(2)$ denote rotation by $\theta$, so that the elements in $A \; SO(2)$ can be written as $AR_\theta$. One can readily check the following identities:
\begin{itemize}
\item $[AR_\theta,R_\tau]=R_{2\tau}$
\item $[R_\theta,AR_\tau]=R_{-2\theta}$
\item $[AR_\theta,AR_\tau]=R_{2(\tau-\theta)}$
\end{itemize}
 The above identities imply that $\Sigma\tilde{c}_*\colon H_2(\Sigma(O(2)\wedge O(2)))\to H_2(\Sigma SO(2))\cong \Z$ maps the fundamental class of each of those four copies of $S^2$ to $\pm 2$. The commutative diagram
\begin{equation}\label{rcomO(2)}
  \begin{gathered}
\xymatrix{
H_2(\Ecom O(2))\ar[r]^{\cm_*}&H_2(BSO(2))\\
H_2(\Sigma(O(2)\wedge O(2)))\ar[r]^-{\Sigma \tilde{c}_*}\ar[u]&H_2(\Sigma SO(2))\ar[u]_\cong
 }
\end{gathered}
\end{equation}
then implies that  $2\Z\subseteq \mathop{\mathrm{im}} \cm_*$.

Recall that Ganea's Lemma says that the homotopy fiber of the inclusion $\Sigma G\to BG$ is the join $G*\Omega BG \simeq G * G \simeq \Sigma(G\wedge G)$. Moreover, the inclusion factors as $\Sigma G \to \Bcom G\to BG$, which induces a map of homotopy fiber sequences:
\[
\xymatrix{
\Sigma (G\wedge G) \ar[r]\ar[d]&\Ecom G \ar[d]\\
\Sigma G \ar[r]\ar[d]&\Bcom G \ar[d]\\
BG\ar@{=}[r]&BG.
}
\]

We may conclude that the map between the fibers is equivariant with respect to the monodromy action. For $G=O(2)$, it was shown in \cite[Corollary 7.3]{lowdim} that the action of $\pi_1(BO(2))=\Z/2$ on $H_2(\Ecom O(2))\cong \Z\oplus \Z$ swaps and negates both entries. The corresponding action on $\Sigma(O(2)\wedge O(2))$ can be deduced from the $O(2)$-conjugation action on the simplicial model of $EO(2)$. Since the action factors through $\pi_0(O(2))=\Z/2$, we only need to know the action of conjugation by the matrix $A$ on $O(2)$. For rotations $R_\theta$, we have $AR_\theta A=R_{-\theta}$ and for reflections $AR_{\theta}$, we have $A(AR_{\theta})A=AR_{-\theta}$. By the above description of the generators in $H_2(\Sigma (O(2)\wedge O(2)))$, we see that the monodromy action negates each of the four generators.

\begin{proof}[Proof of Theorem \ref{adcomO(2)}]
  The commutative square (\ref{rcomO(2)}) is equivariant for the action of $\Z/2$ given by conjugating by $A$.
By the description of the action given above, we conclude that the image of $H_2(\Sigma (O(2)\wedge O(2)))\to H_2(\Ecom O(2)) \xrightarrow{\cong} \Z \oplus \Z$ is contained in the diagonal of $\Z\oplus \Z$; and that the composite  $\Z \oplus \Z \cong H_2(\Ecom O(2)) \xrightarrow{\cm_{*}} H_2(BSO(2)) \cong \Z$ is given by $(u,v) \mapsto k(u+v)$ for some $k\in \Z$. Then, the image of the diagonal under $\cm_*$ is $2k\Z$. This shows that following the top path of diagram (\ref{rcomO(2)}), the image is $2k\Z$. We have already shown that following the bottom path of (\ref{rcomO(2)}), the image is $2\Z$. Therefore $k=\pm 1$, as desired.
\end{proof}

\begin{proposition}\label{EcomOn}
Let $n\geq 3$. Then $\cm\colon \Ecom O(n)\to BSO(n)$ is 3-connected.
\end{proposition}
\begin{proof}
First we show that the map induces an isomorphism on $\pi_2$. Consider the diagram
\[
\xymatrix{
\Ecom O(n)\ar[r]^{\cm}&BSO(n)\\
\Ecom O(2)\ar[r]^{\cm}\ar[u]&BSO(2)\ar[u]
 }
\]
where the vertical arrows are induced by the inclusions $O(2)\to O(n)$ and $SO(2)\to SO(n)$.

Let's identify $\pi_2$ of all these spaces. For $\Ecom O(2) \simeq \Sigma(S^1 \times S^1)$, we have $\pi_2(\Ecom O(2)) \cong \Z \oplus \Z$; the groups $\pi_2(BSO(2)) \cong \Z$ and $\pi_2(BSO(n)) \cong \Z/2$ are well-known. The fact that $\pi_2(\Ecom O(n)) \cong \Z/2$ is implicit in the literature. Indeed, in \cite[Theorem 6.3 and Remark]{Ad5} it is proved that for every $k\geq 2$, $\pi_k(\Bcom G)$ has a splitting as $\pi_k(\Ecom G)\oplus\pi_k(BG)$, induced by a homotopy section of the fibration $\Omega\Ecom G\to \Omega\Bcom G\to \Omega BG$. And it is shown in \cite[Proposition 4.4]{RV} that for $n\geq 3$, we have $\pi_2(\Bcom O(n))=\Z/2\oplus\Z/2$, which then implies that $\pi_2(\Ecom O(n))=\Z/2$.

Because $\pi_2(\Ecom O(n)) \cong \Z/2 \cong \pi_2(BSO(n))$, $\cm_{*}$ can only be zero or an isomorphism. From Theorem \ref{adcomO(2)} we see that the composite starting with the bottom horizontal arrow is surjective on $\pi_2$, so $\cm_{*}$ must be an isomorphism.

Next we deal with $\pi_1$ and $\pi_3$. First, $\Ecom O(n)$ is simply connected by \cite[Lemma 4.3]{Ad3}, since any reflection in $O(n)$ induces a section of $O(n)\to \pi_0(O(n))=\Z/2$. Finally, we have a surjection on $\pi_3$ since $\pi_3(BSO(n))=0$. 
\end{proof}

{\footnotesize {\sc \noindent Omar Antol\'{\i}n Camarena\\
   Instituto de Matem\'aticas, UNAM, Mexico City}\\
  \emph{E-mail address}:
  \href{mailto:omar@matem.unam.mx}{\texttt{omar@matem.unam.mx}}}
\medskip

{\footnotesize {\sc \noindent Bernardo Villarreal \\
   Instituto de Matem\'aticas, UNAM, Mexico City}\\
  \emph{E-mail address}:
  \href{mailto:villarrealr@matem.unam.mx}{\texttt{villarreal@matem.unam.mx}}}
\end{document}